\documentclass[10pt]{article}
\usepackage{amssymb,amsmath,amsthm}
\usepackage{verbatim}
\usepackage{graphicx}
\newtheorem{theorem}{Theorem}

\newtheorem{conjecture}[theorem]{Conjecture}
\newtheorem{claim}{Claim}
\newtheorem{lemma}[theorem]{Lemma}
\newtheorem{corollary}[theorem]{Corollary}
\theoremstyle{remark}
\newtheorem{remark}{Remark}

\begin{document}
\title{Almost isoperimetric subsets of the discrete cube}
\author{David Ellis\footnote{St John's College, Cambridge, UK.}}
\date{May 2010}
\maketitle

\begin{abstract}
We show that a set \(A \subset \{0,1\}^{n}\) with edge-boundary of size at most
\[|A| (\log_{2}(2^{n}/|A|) + \epsilon)\]
can be made into a subcube by at most \((2 \epsilon/\log_{2}(1/\epsilon))|A|\) additions and deletions, provided \(\epsilon\) is less than an absolute constant.

We deduce that if \(A \subset \{0,1\}^{n}\) has size \(2^{t}\) for some \(t \in \mathbb{N}\), and $A$ cannot be made into a subcube by fewer than \(\delta |A|\) additions and deletions, then the edge-boundary of $A$ has size at least
\[|A| \log_{2}(2^{n}/|A|) + |A| \delta \log_{2}(1/\delta) = 2^{t}(n-t+\delta \log_{2}(1/\delta)),\]
provided \(\delta\) is less than an absolute constant. This is sharp whenever \(\delta = 1/2^{j}\) for some \(j \in \{1,2,\ldots,t\}\).
\end{abstract}

\section{Introduction}
We work in the \(n\)-dimensional discrete cube \(\{0,1\}^{n}\), the set of all 0-1 vectors of length \(n\). This may be identified with \(\mathcal{P}([n])\), the set of all subsets of \([n] = \{1,2,\ldots,n\}\), by identifying a set \(x \subset [n]\) with its characteristic vector \(\chi_{x}\) in the usual way. A \(d\)-\emph{dimensional subcube} of \(\{0,1\}^{n}\) is a set of the form
\[\{x \in \{0,1\}^{n}:\ x_{i_{1}}=a_{1},x_{i_{2}}=a_{2},\ldots,x_{i_{n-d}}=a_{n-d}\},\]
where \(i_{1}<i_{2}< \ldots < i_{n-d}\) are coordinates, and \(a_{1},a_{2},\ldots\) and \(a_{n-d}\) are fixed elements of \(\{0,1\}\). The coordinates \(i_{1},i_{2},\ldots,i_{n-d}\) are called the {\em fixed} coordinates; the other coordinates are called the {\em moving} coordinates, and \(n-d\) is called the {\em codimension} of the subcube.

Consider the graph \(Q_{n}\) with vertex-set \(\{0,1\}^{n}\), where we join two 0-1 vectors if they differ in exactly one coordinate; this graph is called the \(n\)-{\em dimensional hypercube}. Given a set \(A \subset \{0,1\}^{n}\), the {\em edge-boundary} of \(A\) is defined to be the set of all edges of \(Q_{n}\) joining a point in \(A\) to a point not in \(A\). We write \(\partial A\) for the edge-boundary of \(A\).

For \(1 \leq k \leq 2^{n}\), let \(C_{n,k}\) be the first \(k\) elements of the {\em binary ordering} on \(\mathcal{P}([n])\), defined by
\[x < y \Leftrightarrow \max(x \Delta y) \in y.\]

The edge-isoperimetric inequality of Harper \cite{harper}, Lindsey \cite{lindsey}, Bernstein \cite{bernstein} and Hart \cite{hart} states that among all subsets of \(\{0,1\}^{n}\) of size \(k\), \(C_{n,k}\) has the smallest possible edge-boundary.

A slightly weaker form is as follows:

\begin{equation}
\label{eq:edgeiso}
|\partial A| \geq |A|\log_{2}(2^{n}/|A|)\quad \forall A \subset \{0,1\}^{n};
\end{equation}
equality holds if and only if \(A\) is a subcube. We call \(|\partial A|/|A|\) the {\em average out-degree} of \(A\); (\ref{eq:edgeiso}) says that the average out-degree of \(A\) is at least \(\log_{2}(2^{n}/|A|)\) (which is the average out-degree of a subcube of size \(|A|\), when \(|A|\) is a power of 2). Writing \(p = |A|/2^{n}\) for the measure of the set \(A\), we may rewrite (\ref{eq:edgeiso}) as:
\[|\partial A| \geq 2^{n}p\log_{2}(1/p)\quad \forall A \subset \{0,1\}^{n}.\]

Hence, if \(|A|=2^{n-1}\), \(|\partial A| \geq 2^{n-1}\), and equality holds only if \(A\) is a codimension-1 subcube, in which case the edge-boundary consists of all the edges in one direction.

It is natural to ask whether it is always possible to find a direction in which there are many boundary edges. For \(i \in [n]\), we write
\[A_{i}^{+} = \{x \setminus \{i\}: x \in A,\ i \in x\} \subset \mathcal{P}([n]\setminus \{i\}),\]
and
\[A_{i}^{-}=\{x \in A: i \notin x\} \subset \mathcal{P}([n]\setminus \{i\});\]
\(A_{i}^{+}\) and \(A_{i}^{-}\) are called the {\em upper} and {\em lower} \(i\)-{\em sections} of \(A\), respectively. We write
\[\partial_{i} A = |A_{i}^{+} \Delta A_{i}^{-}|\]
for the number of edges of the boundary of \(A\) in direction \(i\). The \emph{influence} of the coordinate \(i\) on the set \(A\) is defined to be
\[\beta_{i} = |A_{i}^{+} \Delta A_{i}^{-}|/2^{n-1},\]
i.e. the fraction of direction-\(i\) edges of \(Q_{n}\) which belong to \(\partial A\). This is simply the probability that if \(S \subset \mathcal{P}([n])\) is chosen uniformly at random, \(A\) contains exactly one of \(S\) and \(S \Delta \{i\}\).

Clearly, we have $\sum_{i=1}^{n} \beta_i = |\partial A|/2^{n-1}$. The quantity $\sum_{i=1}^{n} \beta_i$ is sometimes called the {\em total influence}.

Ben-Or and Linial \cite{benor} conjectured that for any set \(A \subset \{0,1\}^{n}\) with \(|A|=2^{n-1}\), there exists a coordinate with influence at least \(\Omega(\tfrac{\log_{2} n}{n})\). This was proved by Kahn, Kalai and Linial; it follows from the celebrated KKL Theorem:
\begin{theorem}[Kahn, Kalai, Linial \cite{tribeskkl}]
\label{thm:kkl}
If \(A \subset \{0,1\}^{n}\) with measure \(p\), then
\[\sum_{i=1}^{n} \beta_{i}^{2} \geq Cp^{2}(1-p)^{2} (\ln n)^{2}/n,\]
where \(C>0\) is an absolute constant.
\end{theorem}
\begin{corollary}
\label{corr:oneinfluencelarge}
If \(A \subset \{0,1\}^{n}\) with measure \(p\), then there exists a coordinate \(i \in [n]\) with
\[\beta_{i} \geq C'p(1-p) (\ln n)/n,\]
where \(C'>0\) is an absolute constant.
\end{corollary}
Corollary \ref{corr:oneinfluencelarge} is sharp up to the value of the absolute constant \(C'\), as can be seen from the `tribes' construction of Ben-Or and Linial \cite{benor}. Let \(n = kl\), and split \([n]\) into \(l\) `tribes' of size \(k\). Let \(A\) be the set of all 0-1 vectors which are identically 0 on at least one tribe. Observe that
\[|A| = (1-(1-2^{-k})^{l})2^{n},\]
\[|\partial A| = n2^{n-k} (1-2^{-k})^{l-1},\]
and
\[\beta_{i} = 2^{-(k-1)} (1-2^{-k})^{l-1}\quad \forall i \in [n].\]
Let \(k=2^{j}\) for some \(j \in \mathbb{N}\), and let \(l = 2^{k}/k\), so that \(n = 2^{k} = 2^{2^{j}}\); then
\[1-p = (1-2^{-k})^{l} = (1-2^{-k})^{2^{k}/k} = 1-1/k+O(1/k^{2}),\]
and
\[\beta_{i} = \frac{2(1-p)}{n(1-2^{-k})} = \frac{2(1-1/k+O(1/k^{2}))}{n}\quad \forall i \in [n],\]
so
\[\frac{\beta_{i}}{p(1-p) \ln(n)/n} = \frac{2(1-1/k+O(1/k^{2}))}{(1/k-O(1/k^{2})(1-O(1/k)) k \ln 2} = \frac{2}{\ln 2}(1+O(1/k)).\]
The best possible values of the constants \(C\) and \(C'\) (in Theorem \ref{thm:kkl} and Corollary \ref{corr:oneinfluencelarge} respectively) remain unknown. Falik and Samorodnitsky \cite{falik} have shown that one can take \(C = 4\), and therefore \(C' = 2\).

Kahn, Kalai and Linial's proof of Theorem \ref{thm:kkl} is one of the first instances of Fourier analysis on \(\{0,1\}^{n}\) being used to prove a purely combinatorial result; Fourier analysis has since become a very important tool in both probabilistic and extremal combinatorics. More recently, Falik and Samorodnitsky \cite{falik} gave an entirely combinatorial proof of Theorem \ref{thm:kkl}; a similar proof was found independently by Rossignol \cite{rossignol}.

In \cite{friedgutjuntatheorem}, Friedgut considers the problem of determining the structure of subsets of $\{0,1\}^n$ with edge-boundary of size at most $K2^{n-1}$, or equivalently, with total influence at most $K$, where $K$ is a constant (or a slowly-growing function of $n$). Using the Fourier-analytic machinery of \cite{tribeskkl}, Friedgut proved the following.

\begin{theorem}[Friedgut's `Junta' theorem]
Let $A \subset \{0,1\}^n$, and suppose that $|\partial A| \leq K2^{n-1}$. Then there exists $B \subset \{0,1\}^n$ such that $|A \Delta B| \leq \epsilon 2^n$, and $B$ is a $\lfloor 2^{C_0 K/\epsilon} \rfloor$-junta, where $C_0$ is an absolute constant.
\end{theorem}

Here, if $B \subset \{0,1\}^n$, and $j \in \mathbb{N}$, we say that $B$ is a {\em j-junta} if there exists a set of coordinates $J \subset [n]$ such that $|J| \leq j$, and {\em the event $\{x \in B\}$ depends only upon the values $(x_j)_{j \in J}$}. The condition in italics is of course equivalent to saying that $B$ is a union of subcubes which all have $J$ as their set of fixed coordinates, or that the characteristic function of $B$ depends only upon the cooordinates in $J$.

Freidgut's theorem is sharp up to the value of the absolute constant $C_0$, as can be seen by taking the set $A$ to be a product of a subcube of codimension $\lfloor \log_2(1/(6\epsilon)) \rfloor$, with a set defined by the `tribes' construction above.

In \cite{falik}, Falik and Samorodnitsky use influence-based methods to obtain several other results on subsets of \(\{0,1\}^{n}\) with small edge-boundary.

In this paper, we will investigate the structure of subsets $A \subset \{0,1\}^n$ whose the edge-boundary has size somewhat closer to \(|A| \log_{2}(2^{n}/|A|)\). In particular, we will try to determine how small the edge-boundary must be, to guarantee that $A$ is close in structure to a single subcube. This question has already been investigated by several researchers. Using the techniques of Fourier analysis, Friedgut, Kalai and Naor \cite{fkn} proved that if \(A \subset \{0,1\}^{n}\) with \(|A|=2^{n-1}\) and \(|\partial A| \leq 2^{n-1}(1+\epsilon)\), then \(A\) can be made into a codimension-1 subcube by at most \(K \epsilon 2^{n-1}\) additions and deletions, where \(K\) is an absolute constant. Bollob\'as, Leader and Riordan \cite{leader} conjectured that for any \(N \in \mathbb{N}\), there exists a constant \(K_{N}\) depending on \(N\) such that any \(A \subset \{0,1\}^{n}\) with \(|A|=2^{n-N}\) and
\[|\partial A| \leq (1+\epsilon)|A|\log_{2}(2^{n}/|A|)\]
can be made into a codimension-\(N\) subcube by at most \(K_{N} \epsilon 2^{n-N}\) additions and deletions. They proved this for \(N=2\) and \(N=3\), also using the techniques of Fourier analysis. We remark that \(K_{N}\) must necessarily depend on \(N\). Indeed, as was observed by Samorodnitsky \cite{alextalk}, a variant of the `tribes' construction of Ben-Or and Linial provides an example of a (small) set \(A\) satisfying
\[|\partial A| \leq (1+\epsilon) |A|\log_{2}(2^{n}/|A|),\]
and yet requiring at least \((1-o(1))|A|\) additions and deletions to make it into a subcube. As above, let \(n = kl\), split \([n]\) into \(l\) `tribes' of size \(k\), and let \(A\) be the set of all 0-1 vectors which are identically 0 on at least one tribe. Fix an integer \(s\). Let \(k = 2^{j}\), and let \(l = 2^{k/2^{s}}/k = 2^{2^{j-s}-j}\), so that \(n = 2^{k/2^{s}} = 2^{2^{j-s}}\). Let \(j \to \infty\). Then
\[1-p = (1-2^{-k})^{l} = 1-l2^{-k}+O((l2^{-k})^{2}) \geq 1-l2^{-k},\]
so
\[p \leq l2^{-k},\]
and therefore
\[\log_{2}(1/p) \geq k-\log_{2}l = (1-2^{-s})k+\log_{2}k.\]
Note that
\[|\partial A| = n2^{n-k} (1-2^{-k})^{l-1} = \frac{n2^{n-k}(1-p)}{1-2^{-k}} = n2^{n-k}(1+O(l2^{-k})).\]
Hence,
\begin{eqnarray*}
\frac{|\partial A|}{|A|\log_{2}(2^{n}/|A|)} & \leq & \frac{n2^{n-k}(1+O(l2^{-k}))}{(l2^{-k}(1-O(l2^{-k})))((1-2^{-s})k+\log_{2}k)2^{n}}\\
& = & \frac{kl(1+O(l2^{-k}))}{l((1-2^{-s})k+\log_{2}k)}\\
& = & \frac{1+O(l2^{-k})}{1-2^{-s}+(\log_{2}k)/k}\\
& < & \frac{1}{1-2^{-s}},
\end{eqnarray*}
provided \(j\) is sufficiently large depending on \(s\). For any \(\epsilon > 0\), this can clearly be made \(\leq 1+\epsilon\) by choosing \(s\) to be sufficiently large depending on \(\epsilon\). However, \(A\) is a union of \(l\) codimension-\(k\) subcubes with disjoint sets of fixed coordinates, and therefore requires at least \((1-o(1))|A|\) additions and deletions to make it into a subcube.

Samorodnitsky \cite{alextalk} conjectured that given any \(\delta >0\), there exists an \(a > 0\) such that any \(A \subset \{0,1\}^{n}\) with
\[|\partial A| \leq (1+a/n)|A|\log_{2}(2^{n}/|A|)\]
can be made into a subcube by at most \(\delta|A|\) additions and deletions. Making use of a result of Keevash \cite{keevash} on the structure of \(r\)-uniform hypergraphs with small shadows, he proved that any \(A \subset \{0,1\}^{n}\) with
\[|\partial A| \leq (1+n^{-4})|A|\log_{2}(2^{n}/|A|)\]
can be made into a subcube by at most \(o(|A|)\) additions and deletions.

It turns out that the correct condition to ensure that \(A\) is close to a subcube is that \(|\partial A|/|A|\), the average out-degree of \(A\), is close to \(\log_{2}(2^{n}/|A|)\). Our first main result (Theorem \ref{thm:moreprecisestability}) implies that if \(A \subset \{0,1\}^{n}\) has edge-boundary of size at most
\begin{equation}
 \label{eq:almostisoperimetricset}
|A| (\log_{2}(2^{n}/|A|) + \epsilon),
\end{equation}
where \(\epsilon\) is less than an absolute constant, then it can be made into a subcube by at most \[(1+O(1/\log_{2}(1/\epsilon)))\frac{\epsilon}{\log_{2}(1/\epsilon)}|A| \leq \frac{2\epsilon}{\log_{2}(1/\epsilon)}|A|\]
additions and deletions. This proves the above conjecture of Bollob\'as, Leader and Riordan, and also that of Samorodnitsky.

We then prove Theorem \ref{thm:exactstabilitysubcubes}, which states that if \(A \subset \{0,1\}^{n}\) has size \(2^{t}\) for some \(t \in \mathbb{N}\), and edge-boundary of size at most
\[|A| (\log_{2}(2^{n}/|A|) + \epsilon) = 2^{t}(n-t+\epsilon),\]
where \(\epsilon\) is less than an absolute constant, then it can be made into a \(t\)-dimensional subcube by at most \(\delta_{1}(\epsilon) |A|\) additions and deletions, where \(\delta_{1}(\epsilon)\) is the unique root of
\[x \log_{2}(1/x) = \epsilon\]
in \((0,1/e)\). It follows that if \(A \subset \{0,1\}^{n}\) has size \(2^{t}\) for some \(t \in \mathbb{N}\), and cannot be made into a subcube by fewer than \(\delta |A|\) additions and deletions, then
\[|\partial A| \geq |A| \log_{2}(2^{n}/|A|) + |A| \delta \log_{2}(1/\delta) = 2^{t}(n-t+\delta \log_{2}(1/\delta)),\]
provided \(\delta\) is less than an absolute constant. This is sharp whenever \(\delta = 1/2^{j}\) for some \(j \in \{1,2,\ldots,t\}\).

Our first aim is to prove a `rough' stability result (Theorem \ref{thm:stability}), stating that if \(A\) is `almost isoperimetric', in the sense that the average out-degree of \(\partial A\) is not too far above \(\log_{2}(2^{n}/|A|)\), then \(A\) can be made into a subcube by a small number of additions and deletions. Influence-based methods play a crucial role in our proof. Indeed, it will turn out that a set \(A \subset \{0,1\}^{n}\) satisfying (\ref{eq:almostisoperimetricset}) must have each influence either very small or very large. We will use the following theorem of Talagrand \cite{talagrand}:

\begin{theorem}[Talagrand]
\label{thm:talagrand}
Suppose \(A \subset \{0,1\}^{n}\) with measure
\[\frac{|A|}{2^{n}} = p;\]
then its influences satisfy:
\[\sum_{i=1}^{n}\beta_{i}/\log_{2}(1/\beta_{i}) \geq Kp(1-p),\]
where \(K > 0\) is an absolute constant.
\end{theorem}

This implies that if all the influences are small, the edge-boundary must be very large. This will help to show that there must be a coordinate, \(i\) say, of very large influence. It will follow that one of the \(i\)-sections of \(A\) is very small. An inductive argument will enable us to complete the proof.

\section{Main results}
We first prove a sequence of results on the rough structure of subsets of \(\{0,1\}^n\) with small edge-boundary. If \(A \subset \{0,1\}^n\), and \(i \in [n]\), we define
\[\gamma_i = \frac{\min\{|A_i^{+}|,|A_{i}^{-}|\}}{|A|}.\]
(Observe that we always have \(\gamma_i \leq 1/2\).) We first show that if \(A \subset \{0,1\}^n\) has small edge-boundary, then for each \(i \in [n]\), either one of the \(i\)-sections of \(A\) is very small, or else the upper and lower \(i\)-sections of \(A\) have very similar sizes.

\begin{lemma}
\label{lemma:cases}
Let \(A \subset \{0,1\}^n\) with
\begin{equation}
\label{eq:smallboundary}
|\partial A| = |A| (\log_{2}(2^{n}/|A|)+\epsilon_0).
\end{equation}
Then for each \(i \in [n]\), either
\begin{enumerate}
 \item \(\gamma_{i} \leq \epsilon_{0}/(5(\log_{2}5-2))\), or
\item \(1/2 - \epsilon_{0} < \gamma_{i} \leq 1/2\).
\end{enumerate}
\end{lemma}
\begin{proof}
Let \(A \subset \{0,1\}^{n}\) satisfying the hypothesis of the lemma. Write
\[p = \frac{|A|}{2^{n}}\]
for the measure of \(A\); then
\[|\partial A| = 2^n p (\log_{2}(1/p)+\epsilon_0).\]
Fix \(i \in [n]\). Without loss of generality, we may assume that \(|A_{i}^{+}| \leq |A_{i}^{-}|\), so
\[\gamma_{i} = \frac{|A_{i}^{+}|}{|A|}.\]
Write \(\gamma = \gamma_i\). Let
\[p^{+} = \frac{|A_{i}^{+}|}{2^{n-1}},\ p^{-}= \frac{|A_{i}^{-}|}{2^{n-1}};\]
note that
\[p^{+} = 2 \gamma p,\ p^{-} =2(1-\gamma)p.\]

Define \(\epsilon^{+},\epsilon^{-}\) by
\[|\partial A_{i}^{+}| = |A_{i}^{+}|(\log_{2}(2^{n-1}/|A_{i}^{+}|) + \epsilon^{+}),\quad |\partial A_{i}^{-}| = |A_{i}^{-}|(\log_{2}(2^{n-1}/|A_{i}^{-}|) + \epsilon^{-}).\]
Observe that
\begin{eqnarray}
\label{eq:indbound}
|\partial A| & = & |\partial A_{i}^{+}|+|\partial A_{i}^{-}| + |A_{i}^{+} \Delta A_{i}^{-}| \nonumber \\
& = & |A_{i}^{+}|(\log_{2}(2^{n-1}/|A_{i}^{+}|)+ \epsilon^{+}) + |A_{i}^{-}|(\log_{2}(2^{n-1}/|A^{-}|) + \epsilon^{-}) + |A_{i}^{+} \Delta A_{i}^{-}| \nonumber \\
& = & \gamma |A| \log_{2}(2^{n}/(2\gamma|A|))+(1-\gamma)|A|(\log_{2}(2^{n}/(2(1-\gamma)|A|) + \epsilon^{+} |A_{i}^{+}| +\epsilon^{-} |A_{i}^{-}|\nonumber \\
&& + |A_{i}^{+} \Delta A_{i}^{-}| \nonumber \\
& = & |A| \log_{2}(2^{n}/|A|)-(1-H_{2}(\gamma))|A| + \epsilon^{+} |A_{i}^{+}| +\epsilon^{-} |A_{i}^{-}| + |A_{i}^{+} \Delta A_{i}^{-}|\\
& \geq & |A| \log_{2}(2^{n}/|A|)-(1-H_{2}(\gamma))|A| + \epsilon^{+} |A_{i}^{+}| +\epsilon^{-} |A_{i}^{-}| + \left||A_{i}^{+}|-|A_{i}^{-}|\right| \nonumber \\
& = & |A| \log_{2}(2^{n}/|A|)-(1-H_{2}(\gamma))|A| + \epsilon^{+} |A_{i}^{+}| +\epsilon^{-} |A_{i}^{-}|+ (1-2\gamma)|A| \nonumber \\
& = & |A| \log_{2}(2^{n}/|A|)+(H_{2}(\gamma)-2\gamma)|A| + \epsilon^{+} |A_{i}^{+}| +\epsilon^{-} |A_{i}^{-}| \nonumber\\
& = & |A| \log_{2}(2^{n}/|A|)+F(\gamma)|A| + \epsilon^{+} |A_{i}^{+}| +\epsilon^{-} |A_{i}^{-}|,\nonumber
\end{eqnarray}
where \(H_{2}:[0,1] \to \mathbb{R}\) denotes the {\em binary entropy} function,
\[H_{2}(\gamma) := \gamma \log_{2}(1/\gamma)+(1-\gamma) \log_{2}(1/(1-\gamma)),\]
and
\[F(\gamma):=H_{2}(\gamma)-2\gamma.\]
Hence, (\ref{eq:smallboundary}) implies that
\begin{equation}
 \label{eq:inductiveiso}
\gamma \epsilon^{+} + (1-\gamma) \epsilon^{-} + F(\gamma) \leq \epsilon_{0}.
\end{equation}
Therefore, crudely,
\[F(\gamma) \leq \epsilon_{0}.\]
The function \(F\) is concave on \([0,1/2]\), and attains its maximum at \(\gamma = 1/5\), where it takes the value \(\log_{2}5-2\). Hence, for \(\gamma \leq 1/5\),
\[F(\gamma) \geq 5(\log_{2}5-2)\gamma,\]
whereas for \(1/5 \leq \gamma \leq 1/2\),
\[F(1/2-\eta) \geq \tfrac{10}{3} (\log_{2}5-2)\eta > \eta.\]
Hence, for each \(i \in [n]\), either
\begin{enumerate}
 \item \(\gamma_{i} \leq \epsilon_{0}/(5(\log_{2}5-2))\), or
\item \(1/2 - \epsilon_{0} < \gamma_{i} \leq 1/2\),
\end{enumerate}
proving the lemma.
\end{proof}
\begin{remark}
\label{remark:influences}
We can of course rephrase the conclusion of Lemma \ref{lemma:cases} in terms of influences. Let \(A \subset \{0,1\}^n\) satisfying (\ref{eq:almostiso}). Observe that if case 1 occurs for \(i \in [n]\), then
\begin{equation}
 \label{eq:largeinfluence}
\beta_{i} \geq (1-2\gamma_{i}) |A|/2^{n-1} = 2(1-2\gamma_{i})p \geq 2\left(1-2\frac{\epsilon_0}{5(\log_{2}5-2)}\right)p,
\end{equation}
---the \(i\)th influence is `large'.

If, on the other hand, case 2 occurs, then by (\ref{eq:indbound}), we have
\[|A_{i}^{+} \Delta A_{i}^{-}| \leq |\partial A| - |A| \log_{2}(2^{n}/|A|)+(1-H_{2}(\gamma_{i}))|A| = (\epsilon_{0} + 1-H_{2}(\gamma_{i}))|A|.\]
Since \(H_{2}\) is concave, with \(H_{2}(1/2) = 1\), we have
\[1-H_{2}(1/2 - \eta) \leq 2\eta\ (0 \leq \eta \leq 1/2),\]
and therefore
\[|A_{i}^{+} \Delta A_{i}^{-}| < 3\epsilon_{0} |A|,\]
i.e.
\[\beta_{i} < 6\epsilon_{0} p,\]
---the \(i\)th influence is `small'.
\end{remark}

We now show that if the edge-boundary of \(A\) is sufficiently small, then case 1 in Lemma \ref{lemma:cases} must occur for some \(i \in [n]\).

\begin{lemma}
\label{lemma:case1mustoccur}
There exists an absolute constant \(c>0\) such that the following holds. If \(\epsilon \leq c\), and \(A \subset \{0,1\}^{n}\) with measure
\[\frac{|A|}{2^{n}} \leq 1-\epsilon,\]
and
\begin{equation}
 \label{eq:almostiso}
|\partial A| \leq |A| (\log_{2}(2^{n}/|A|)+\epsilon);
\end{equation}
then case 1 must occur for some \(i \in [n]\), i.e. \(\gamma_{i} \leq \epsilon/(5(\log_{2}5-2))\) for some \(i \in [n]\).
\end{lemma}
\begin{proof}
We can easily prove the lemma for sets with measure \(p \in [1/2,7/8]\). Suppose \(A \subset \{0,1\}^n\) has measure \(p \in [1/2,7/8]\) and satisfies (\ref{eq:almostiso}). Suppose for a contradiction that case 2 occurs for every \(i \in [n]\). Then by Remark \ref{remark:influences}, \(\beta_{i} < 6\epsilon p\) for every \(i \in [n]\), and therefore by Theorem \ref{thm:talagrand},
\[\sum_{i=1}^{n} \beta_{i} > Kp(1-p) \log_{2} \left(\frac{1}{6\epsilon p}\right).\]
The right-hand side is at least
\[2p (\log_{2}(1/p)+\epsilon)\]
provided
\[\tfrac{K}{8} \log_{2}\left(\frac{1}{6\epsilon}\right) \geq 2(1+\epsilon),\]
which holds for all \(\epsilon \leq c: = 2^{-32K}/6\). This contradicts (\ref{eq:almostiso}), proving the lemma for \(p \in [1/2,7/8]\).

Now observe that {\em any} set \(A \subset \{0,1\}^n\) with measure \(p \in [7/8,1-\epsilon]\) has
\begin{equation}
\label{eq:toolarge}
|\partial A| > |A| (\log_{2}(2^{n}/|A|)+\epsilon),
\end{equation}
To see this, just apply the edge-isoperimetric inequality (\ref{eq:edgeiso}) to \(A^{c}\):
\[|\partial A|=|\partial (A^{c})| \geq 2^{n} (1-p) \log_{2}(1/(1-p)).\]
It is easily checked that
\[2^{n} (1-p) \log_{2}(1/(1-p)) > 2^{n} p (\log_{2}(1/p)+1-p)\quad \forall p \geq 7/8,\]
so (\ref{eq:toolarge}) holds for all \(p \in [7/8,1-\epsilon]\). Hence, any set \(A \subset \{0,1\}^n\) satisfying (\ref{eq:almostiso}) must have measure \(p \leq 7/8\).

It remains to prove the lemma for all sets of measure \(p \leq 1/2\). Suppose \(A\) has measure \(p \leq 1/2\) and satisfies (\ref{eq:almostiso}). Suppose for a contradiction that case 2 occurs for every \(i \in [n]\).

Fix any \(i \in [n]\). Without loss of generality, we may assume that \(|A_{i}^{+}| \leq |A_{i}^{-}|\), so that
\[\gamma_{i} = \frac{|A_{i}^{+}|}{|A|}.\]
Write \(\gamma = \gamma_i\). Define \(\epsilon^{+}\) and \(\epsilon^{-}\) as in the proof of Lemma \ref{lemma:cases}. By (\ref{eq:inductiveiso}), we have
\[\gamma \epsilon^{+} + (1-\gamma) \epsilon^{-} + F(\gamma) \leq \epsilon.\]
Hence, crudely,
\[\gamma \epsilon^{+} + (1-\gamma) \epsilon^{-} \leq \epsilon,\]
so either \(\epsilon^{+} \leq \epsilon\) or \(\epsilon^{-} \leq \epsilon\).

If \(\epsilon^{+} \leq \epsilon\), then let \(A' = A_{i}^{+}\). The set \(A'\) is a subset of \(\mathcal{P}([n] \setminus \{i\})\) of measure \(p' := 2\gamma p \in ((1-2\epsilon)p,p) \subset [0,1/2]\), satisfying the conditions of the lemma.

If \(\epsilon^{-} \leq \epsilon\), then let \(A' =A_{i}^{-}\); the set \(A'\) is a subset of \(\mathcal{P}([n] \setminus \{i\})\) of measure \(p': = 2(1-\gamma) p < 2(1/2 + \epsilon)p \leq 1/2+\epsilon < 7/8\), satisfying the conditions of the lemma.

If \(A'\) has case 1 occurring for some \(j\), then by (\ref{eq:largeinfluence}),
\begin{eqnarray*}
\beta_{j}' & \geq & 2\left(1-2\frac{\epsilon}{5(\log_{2}5-2)}\right)p'\\
& \geq & 2\left(1-2\frac{\epsilon}{5(\log_{2}5-2)}\right)(1-2\epsilon)p\\
& > & 2(1-2\epsilon)^{2}p,
\end{eqnarray*}
and therefore
\[\beta_{j} > (1-2\epsilon)^{2} p > 6 \epsilon p,\]
contradicting our assumption that \(A\) has case 2 occurring for every \(i \in [n]\). Therefore, \(A'\) also has case 2 occurring for every coordinate. Hence, it must have measure \(p' < 1/2\), by the above argument for sets of measure in \([1/2,7/8]\). Repeat the same argument for \(A'\), and continue; we obtain a sequence of set systems \((A^{(l)})\) on ground sets of sizes \(n-l\), all with measure \(< 1/2\), satisfying the conditions of the lemma, and with case 2 occurring for every coordinate. Stop at the minimum \(M\) such that \(A^{(M)} = \emptyset\); clearly, \(M \leq n-1\). Then \(A^{(M-1)}\) has one of its \(j\)-sections empty for some \(j\), so case 1 must occur for this \(j\), a contradiction. This proves the lemma.
\end{proof}

We can now prove a rough stability result for subsets of \(\{0,1\}^n\) with small edge-boundary:
\begin{theorem}
\label{thm:stability}
There exists an absolute constant \(c>0\) such that if \(A \subset \{0,1\}^{n}\) with
\[|\partial A| \leq |A| \log_{2}(2^{n}/|A|) + \epsilon |A|,\]
for some \(\epsilon \leq c\), then
\[|A \Delta C|/|A| < 3 \epsilon\]
for some subcube \(C\).
\end{theorem}
\begin{proof}
Let \(c\) be the constant in Lemma \ref{lemma:case1mustoccur}. Let \(A \subset \{0,1\}^n\) be such that
\[|\partial A| \leq |A| \log_{2}(2^{n}/|A|) + \epsilon |A|,\]
for some \(\epsilon \leq c\). Let \(\epsilon_0 \leq \epsilon\) be such that
\[|\partial A| = |A| (\log_{2}(2^{n}/|A|)+\epsilon_0).\]
By Lemma \ref{lemma:case1mustoccur}, there exists \(i \in [n]\) with case 1 occurring, i.e. with
\[\gamma_{i} \leq \epsilon/(5(\log_{2}5-2)).\]
Without loss of generality, we may assume that \(i = n\), and that \(|A_{n}^{+}| \leq |A_{n}^{-}|\). In keeping with our earlier notation, we write \(\gamma = \gamma_n = |A_{n}^{+}|/|A|\).

To avoid confusion, we now write \(B^{(0)} = A,\ p^{(0)} = p,\ \epsilon^{(0)} = \epsilon_{0}\), and \(\gamma^{(0)} = \gamma\). Let \(B^{(1)} = A_{n}^{-} \subset \mathcal{P}([n-1])\), let \(p^{(1)} = p^{-}_{n}\), and let \(\epsilon^{(1)} = \epsilon_{n}^{-}\).

By (\ref{eq:inductiveiso}), we have
\[(1-\gamma^{(0)}) \epsilon^{(1)} + F(\gamma^{(0)}) \leq \epsilon^{(0)}.\]
Since \(F(\gamma^{(0)}) \geq 5(\log_{2}5-2) \gamma^{(0)}\), we have
\[(1-\gamma^{(0)}) \epsilon^{(1)} + 5(\log_{2}5-2) \gamma^{(0)} \leq \epsilon^{(0)};\]
it follows that \(\epsilon^{(1)} \leq \epsilon \leq c\). Hence, \(B^{(1)} \subset \mathcal{P}([n-1])\) also satisfies the hypothesis of Theorem \ref{thm:stability} (with \(n\) replaced by \(n-1\)). Its measure \(p^{(1)}\) satisfies
\begin{eqnarray*}
p^{(1)} & = & 2(1-\gamma^{(0)})p^{(0)}\\
& \geq & 2\left(1-\frac{\epsilon^{(0)}}{5(\log_{2}5-2)}\right)p^{(0)}\\
& > & 2(1-\epsilon^{(0)})p^{(0)}\\
& \geq & 2(1-c)p^{(0)}.
\end{eqnarray*}

Repeat the same argument for \(B^{(1)}\). We obtain a sequence of set systems \((B^{(k)})\) on ground sets of sizes \(n-k\), satisfying the hypotheses of Theorem \ref{thm:stability} with \(\epsilon\) replaced by \(\epsilon^{(k)} \leq \epsilon_{0} \leq c\), with measures \(p^{(k)}\) satisfying
\[p^{(k+1)} > 2(1-\epsilon^{(k)})p^{(k)}\quad \forall k \geq 0,\]
and with
\begin{equation}
 \label{eq:inductiveineq}
(1-\gamma^{(k)}) \epsilon^{(k+1)} + F(\gamma^{(k)}) \leq \epsilon^{(k)} \quad \forall k \geq 0.
\end{equation}
Without loss of generality, we may assume that \(B^{(k)} \subset \mathcal{P}([n-k])\).

We may continue this process until we produce a set system \(B^{(N)}\) at stage \(N\), for which \(p^{(N)} > 1-\epsilon_{0}\), at which point we can no longer apply Lemma \ref{lemma:case1mustoccur}. We must now show that \(A\) is close to \(\mathcal{P}([n-N])\). Observe that
\begin{eqnarray*}
|A \setminus B^{(N)}| & = & \sum_{k=0}^{N-1} \gamma^{(k)} p^{(k)} 2^{n-k}\\
& = & \sum_{k=0}^{N-1} 2^{k} \left(\prod_{j < k} (1-\gamma^{(j)})\right)\gamma^{(k)} p_{0}2^{n-k} \\
& = & \sum_{k=0}^{N-1} \left(\prod_{j < k} (1-\gamma^{(j)})\right)\gamma^{(k)} p_{0}2^{n}\\
& = & \sum_{k=0}^{N-1} \left(\prod_{j < k}(1-\gamma^{(j)})\right)\gamma^{(k)} |A|.
 \end{eqnarray*}
By repeatedly applying the inequality (\ref{eq:inductiveineq}), we obtain
\[\sum_{k=0}^{N-1} \left(\prod_{j < k}(1-\gamma^{(j)})\right)F(\gamma^{(k)}) +\left(\prod_{j =0}^{N-1}(1-\gamma^{(j)})\right)\epsilon_{N} \leq \epsilon_{0},\]
so certainly,
\[\sum_{k=0}^{N-1}  \left(\prod_{j < k}(1-\gamma^{(j)})\right)F(\gamma^{(k)}) \leq \epsilon_{0}.\]
Since \(F(\gamma^{(k)}) \geq 5(\log_{2}5-2)\gamma^{(k)}\ (0 \leq k \leq N-1)\), it follows that
\[\sum_{k=0}^{N-1}  \left(\prod_{j < k}(1-\gamma^{(j)})\right)\gamma^{(k)} \leq \frac{\epsilon_{0}}{5(\log_{2}5-2)}.\]
Hence,
\[|A \setminus B^{(N)}| \leq \frac{\epsilon_{0}}{5(\log_{2}5-2)}|A| < \epsilon_{0} |A|.\]
Let \(C = \mathcal{P}([n-N])\), a codimension-\(N\) subcube. Then
\begin{equation}
 \label{eq:AminusCsmall}
|A \setminus C| = |A \setminus B^{(N)}| < \epsilon_{0} |A|.
\end{equation}
Since \(p^{(N)} > 1-\epsilon_{0}\), we have
\begin{equation}
 \label{eq:CminusAsmall}
|C \setminus A| < \epsilon_{0} |C|.
\end{equation}
Hence,
\[|C| < \frac{1}{1-\epsilon_{0}}|A|,\]
and therefore
\[|C \setminus A| < \frac{\epsilon_{0}}{1-\epsilon_{0}}|A| < 2 \epsilon_{0} |A|.\]
Combining this with (\ref{eq:AminusCsmall}) yields:
\begin{equation}
\label{eq:AdifferenceCsmallA}
|A \Delta C| < 3\epsilon_{0}|A|,
\end{equation}
proving Theorem \ref{thm:stability}.
\end{proof}

We may use this rough stability result to obtain a more precise one:

\begin{theorem}
\label{thm:moreprecisestability}
There exists an absolute constant \(c>0\) such that if \(A \subset \{0,1\}^{n}\) with
\[|\partial A| \leq |A| \log_{2}(2^{n}/|A|) + \epsilon |A|,\]
for some \(\epsilon \leq c\), then
\[|A \Delta C| < \delta_{0}(\epsilon) |A|\]
for some subcube \(C\), where \(\delta_{0}(\epsilon)\) is the smallest positive solution of
\[x\log_{2}(1/x) - 3x = \epsilon.\]
\end{theorem}
\begin{proof}
Write
\begin{equation}
 \label{eq:almostisoperimetric2}
|\partial A| = |A| (\log_{2}(2^{n}/|A|) +\epsilon_{0}),
\end{equation}
where \(0 \leq \epsilon_{0} \leq \epsilon\). Choose a subcube \(C\) such that \(|A \Delta C|\) is minimal, and let \(\delta = |A \Delta C|/|A|\). By Theorem \ref{thm:stability}, \(\delta < 3 \epsilon_{0} \leq 3c < 1/2\). 

Without loss of generality, we may assume that \(C = \mathcal{P}([n-N])\). Let \(B = C \setminus A\) and let \(D = A \setminus C\); then
\[|B|+|D| < 3\epsilon_{0} |A|.\]
Since every point of \(D\) is adjacent to at most one point of \(C\), the number of edges in \(\partial A\) between points of \(A \cap C\) and points of \(\{0,1\}^n \setminus C\) is at least
\[N(2^{n-N}-|B|) - |D|.\]
The number of edges in \(\partial A\) between points of \(C\) is at least
\[|B| \log_{2}(2^{n-N}/|B|).\]
Finally, the number of edges of the cube in \(\partial D\) is at least
\[|D|\log_{2}(2^{n}/|D|),\]
and the number of edges of the cube between points in \(D\) and points in \(C\) is at most \(|D|\), so the number of edges of the cube between points of \(D\) and points of \((\{0,1\}^{n} \setminus C)\setminus A\) is at least
\[|D|(\log_{2}(2^{n}/|D|)-1).\]
It follows that
\begin{eqnarray}
\label{eq:tightbound}
|\partial A| & \geq & N(2^{n-N}-|B|) - |D| + |B| \log_{2}(2^{n-N}/|B|) +|D|(\log_{2}(2^{n}/|D|)-1) \nonumber \\
& = & N2^{n-N} + (\log_{2}(2^{n-N}/|B|)-N)|B| +(\log_{2}(2^{n}/|D|)-2)|D|\nonumber \\
& = & N(|A|-|D|+|B|) + (\log_{2}(2^{n-N}/|B|)-N)|B| \nonumber \\
&&+(\log_{2}(2^{n}/|D|)-2)|D| \nonumber \\
& = & N|A| +|B|(\log_{2}(2^{n}/|B|)-N)+|D|(\log_{2}(2^{n}/|D|)-N-2).
\end{eqnarray}
Write \(|B| = \phi |A|\) and \(|D| = \psi|A|\). Then \(\delta = \psi+\phi\). Note that
\[N = \log_{2}\left(\frac{2^{n}}{|A|-|D|+|B|}\right) = \log_{2}\left(\frac{2^{n}}{|A|}\right) - \log_{2} (1-\psi+\phi).\]
Hence, we obtain:
\begin{eqnarray*}
|\partial A| & \geq & |A|\log_{2}(2^{n}/|A|) - |A| \log_{2}(1-\psi+\phi) \\
&& + \phi |A| (\log_{2}(1/\phi) + \log_{2}(1-\psi+\phi))\\
&& + \psi |A| (\log_{2}(1/\psi)-2+\log_{2}(1-\psi+\phi))\\
& = & |A|\log_{2}(2^{n}/|A|) +\\
&& |A|(\phi \log_{2}(1/\phi)+\psi\log_{2}(1/\psi)-2\psi+(\psi+\phi-1)\log_{2}(1-\psi+\phi)) \\
& > & |A|\log_{2}(2^{n}/|A|) + |A|(\psi\log_{2}(1/\psi) + \phi \log_{2}(1/\phi) -3\psi - 3\phi),
\end{eqnarray*}
where the last inequality follows from the fact that \(\psi,\phi < 1/2\).
Observe that the function 
\begin{eqnarray*}
h: (0,1] & \to & \mathbb{R};\\
x & \mapsto & x \log_{2}(1/x)\\
\end{eqnarray*}
is concave, and therefore
\[\psi\log_{2}(1/\psi) + \phi \log_{2}(1/\phi) \geq (\psi+\phi)\log_{2}(1/(\psi+\phi)).\]
We obtain:
\[|\partial A| > |A|\log_{2}(2^{n}/|A|) + |A|((\psi+\phi)\log_{2}(1/(\psi+\phi)) - 3(\psi + \phi)).\]
Hence, by (\ref{eq:almostisoperimetric2}),
\[(\psi+\phi)\log_{2}(1/(\psi+\phi)) -3(\psi + \phi) < \epsilon_{0},\]
i.e.,
\[\delta (\log_{2}(1/\delta)-3) < \epsilon_{0}.\]
It is easy to check that the function
\begin{eqnarray*}
g: (0,1] & \to & \mathbb{R};\\
x & \mapsto & x \log_{2}(1/x) - 3x\\
\end{eqnarray*}
is strictly increasing between \(0\) and \(2^{-(3+1/\ln(2))}\); provided \(3c \leq 2^{-(3+1/\ln(2))}\), it follows that \(\delta < \delta_{0}(\epsilon)\), where \(\delta_{0}(\epsilon)\) is the smallest positive solution of
\[x\log_{2}(1/x) - 3x = \epsilon,\]
proving Theorem \ref{thm:moreprecisestability}.
\end{proof}

\begin{remark}
Observe that
\[\delta_{0}(\epsilon) = (1+O(1/\log_{2}(1/\epsilon)))\frac{\epsilon}{\log_{2}(1/\epsilon)} \leq \frac{2\epsilon}{\log_{2}(1/\epsilon)}.\]
\end{remark}

Similarly, we may obtain an exact stability result for set systems whose size is a power of 2:

\begin{theorem}
 \label{thm:exactstabilitysubcubes}
There exists an absolute constant \(c>0\) such that if \(A \subset \{0,1\}^{n}\) with size \(|A| = 2^{n-N}\) for some \(N \in \mathbb{N}\), and with edge-boundary
\[|\partial A| \leq |A| \log_{2}(2^{n}/|A|) + \epsilon |A|,\]
where \(\epsilon \leq c\), then there exists a codimension-\(N\) subcube \(C\) such that
\[|A \Delta C| \leq \delta_{1}(\epsilon)|A|,\]
where \(\delta_{1}(\epsilon)\) is the unique root of the equation
\[x \log_{2}(1/x) = \epsilon\]
in \((0,1/e)\).
\end{theorem}
\begin{proof}
Write
\begin{equation}
 \label{eq:almostisoperimetric3}
|\partial A| = |A| (\log_{2}(2^{n}/|A|) +\epsilon_{0}).
\end{equation}
where \(0 \leq \epsilon_{0} \leq \epsilon\). Choose a subcube \(C\) such that \(|A \Delta C|\) is minimal, and let \(\delta = |A \Delta C|/|A|\). By Theorem \ref{thm:stability}, \(\delta < 3 \epsilon_{0} \leq 3c < 1/2\).

Suppose \(C\) has codimension \(N'\). Note that if \(N \neq N'\), then \(|A|\) and \(|C|\) would differ by a factor of at least 2, so
\[|A \Delta C| / |A| \geq ||A|-|C||/|A| \geq 1/2,\]
a contradiction. Hence, \(N' = N\), i.e. \(|C|=|A|\).

Let \(B = C \setminus A\); then \(|A  \setminus C| = |C \setminus A| = |B|\). From (\ref{eq:tightbound}), we have
\begin{eqnarray*}
\label{eq:tightboundpowers}
|\partial A| & \geq & |A|\log_{2}(2^{n}/|A|) +|B|(\log_{2}(2^{n}/|B|)-N)+|B|(\log_{2}(2^{n}/|B|)-N-2) \\
& = & |A|\log_{2}(2^{n}/|A|) +2|B|\log_{2}(2^{n}/|B|)-2|B|\log_{2}(2^{n}/|A|)-2|B| \\
& = & |A| \log_{2}(2^{n}/|A|) + |A|\delta \log_{2}(1/\delta).
\end{eqnarray*}
It follows that
\[\delta \log_{2}(1/\delta) \leq \epsilon.\]
Observe that the function 
\begin{eqnarray*}
h: (0,1] & \to & \mathbb{R};\\
x & \mapsto & x \log_{2}(1/x)\\
\end{eqnarray*}
has
\[h'(x) = -\frac{1}{\ln 2}(1+\ln x)\]
and is therefore strictly increasing between \(0\) and \(1/e\), where it attains its maximum of \(1/(e\ln2)\), and strictly decreasing between \(1/e\) and \(1\). Since \(\delta < 3\epsilon \leq 3c < 1/e\), it follows that \(\delta \leq \delta_{1}(\epsilon)\), where \(\delta_{1}(\epsilon)\) is the unique root of the equation
\[x \log_{2}(1/x) = \epsilon\]
in \((0,1/e)\), proving the theorem.
\end{proof}
The following is an immediate consequence of Theorem \ref{thm:exactstabilitysubcubes}:
\begin{corollary}
\label{corr:powersof2}
If \(A \subset \{0,1\}^{n}\) has size \(2^{t}\) for some \(t \in \mathbb{N}\), and cannot be made into a subcube by fewer than \(\delta |A|\) additions and deletions, then its edge-boundary satisfies
\[|\partial A| \geq |A| \log_{2}(2^{n}/|A|) + |A| \max\{\delta \log_{2}(1/\delta),c\} = 2^{t}(n-t+\max\{\delta \log_{2}(1/\delta),c\}),\]
where \(c > 0\) is an absolute constant. There exists an absolute constant \(c'>0\) such that if \(\delta \leq c'\), then
\[|\partial A| \geq |A| \log_{2}(2^{n}/|A|) + |A| \delta \log_{2}(1/\delta) = 2^{t}(n-t+\delta \log_{2}(1/\delta)).\]
\end{corollary}

\begin{remark}
 Observe that all we need from Theorem \ref{thm:stability} to prove Theorem \ref{thm:exactstabilitysubcubes} is that
\[\delta = |A \Delta C|/|A| < 1/e.\]
If we just knew that \(\delta < 1/2\), we could still deduce from the above argument that \(\delta \log_{2}(1/\delta) \leq \epsilon\).
\end{remark}
\begin{remark}
Observe that Theorem \ref{thm:exactstabilitysubcubes} is best possible, apart from the restriction \(\epsilon \leq c\). To see this, let \(C = \mathcal{P}([n-N])\), a codimension-\(N\)-subcube, where \(1 \leq N \leq n-1\). Let \(2 \leq M \leq n-N\), and delete from \(C\) the codimension-\((N+M)\) subcube 
\[B = \{x \cup \{n-N\}:\ x \in \mathcal{P}([n-N-M])\}.\]
Now add on the codimension-\((N+M)\) subcube
\[D = \{x \cup \{n\}:\ x \in \mathcal{P}([n-N-M])\}.\]
The resulting family \(A = (C \setminus B) \cup D\) has
\[|A \Delta C| / |A| = 2^{-(M-1)} \leq 1/2;\]
it is easy to check that all other subcubes \(C' \neq C\) have
\[|A \Delta C'| > |A \Delta C|.\]
Hence,
\[\delta := \min\{|A \Delta C'|:\ C' \textrm{ is a subcube}\}/|A| = |A \Delta C|/|A| = 2^{-(M-1)}.\]
Observe that we have equality in (\ref{eq:tightbound}) for \(A\), and therefore
\[|\partial A| = |A|\log_{2}(2^{n}/|A|)+|A|\delta\log_{2}(1/\delta).\]
\end{remark}

\section{Conclusion and Open Problems}
Consider the function
\begin{eqnarray*}
f(\delta) & = & \textrm{inf} \{\frac{|\partial A|-|A|\log_{2}(2^{n}/|A|)}{|A|}:\ n \in \mathbb{N},\ A \subset \{0,1\}^{n},\\ &&|A| \textrm{ is a power of }2,\ |A \Delta C| \geq \delta |A|\ \textrm{for all subcubes }C\}.
\end{eqnarray*}
We have shown that \(f(\delta) = \max(\delta \log_{2}(1/\delta),c)\) when \(\delta = 1/2^{j}\) for some \(j \in \mathbb{N}\), where \(c >0\) is an absolute constant, implying that \(f(2^{-j}) = j 2^{-j}\) for \(j \in \mathbb{N}\) sufficiently large. We conjecture that the restriction on \(j\) could be removed:
\begin{conjecture}
\label{conj:allj}
For any \(j \in \mathbb{N}\),
\[f(2^{-j}) = j 2^{-j}.\]
\end{conjecture}

As observed above, the function
\begin{eqnarray*}
h: (0,1] & \to & \mathbb{R};\\
x & \mapsto & x \log_{2}(1/x)\\
\end{eqnarray*}
is strictly decreasing between \(1/e\) and \(1\), whereas \(f\) is clearly an non-decreasing function of \(\delta\). It would be interesting to determine the behaviour of \(f(\delta)\) for \(1/2 < \delta \leq 1\).

We also conjecture that Talagrand's Theorem (Theorem \ref{thm:talagrand}) holds with \(K=2\). This was independently conjectured by Samorodnitsky \cite{alextalk}. It would be best possible, as can be seen by taking \(A\) to be a \(t\)-dimensional subcube; then \(n-t\) influences are \(2^{-(n-t-1)}\), and the rest are zero, so
\[\sum_{i=0}^{n} \beta_{i} / \log_{2}(1/\beta_{i}) = \frac{(n-t)2^{-(n-t-1)}}{n-t-1}.\]
Hence,
\[\frac{1}{p(1-p)}\sum_{i=0}^{n} \beta_{i} / \log_{2}(1/\beta_{i}) = \frac{2(n-t)}{(n-t-1)(1-2^{-(n-t)})} \to 2 \quad \textrm{as } n \to \infty.\]

Knowing this would obviously weaken the upper bound on \(\epsilon\) required to prove Theorem \ref{thm:stability}, though it would not result in a proof of Conjecture \ref{conj:allj}.\\

It would be interesting to determine the structure of subsets \(A \subset \{0,1\}^{n}\) satisfying
\begin{equation} \label{eq:constant-factor} |\partial A| \leq L |A| \log_{2}(2^{n}/|A|)\end{equation}
for \(L\) a fixed positive constant. It is easy to check that 
$$ k \log_2(2^n/k) \leq |\partial C_{n,k}| \leq 2k \log_2(2^n/k)\quad \forall k \leq 2^{n-1},$$
so when $|A| \leq 2^{n-1}$, condition (\ref{eq:constant-factor}) is equivalent to saying that the edge-boundary of $A$ is within a constant factor of the minimum. Regarding this case, Kahn and Kalai \cite{kahnkalaiconjecture} make the following conjecture.

\begin{conjecture}[Kahn, Kalai]
\label{conj:kk}
For any \(L > 0\), there exist \(L' > 0\) and \(\delta >0\) such that the following holds. If \(A \subset \{0,1\}^{n}\) is monotone increasing, with measure $p = \frac{|A|}{2^{n}} \leq 1/2$, and with edge-boundary satisfying
\[|\partial A| \leq L |A| \log_{2}(2^{n}/|A|),\]
then there exists a subcube \(C \subset \{0,1\}^n\) with codimension at most \(L' \log_{2}(1/p)\) and all fixed coordinates equal to 1, such that
 $$\frac{|A \cap C|}{|C|} \geq (1+\delta)p.$$
\end{conjecture}

We believe Conjecture \ref{conj:kk} to be true for non-monotone sets as well, if one allows the subcube \(C\) to have fixed \(0\)'s as well as fixed \(1\)'s:

\begin{conjecture}
For any \(L > 0\), there exist \(L' > 0\) and \(\delta >0\) such that the following holds. If \(A \subset \{0,1\}^{n}\) has measure $p = \frac{|A|}{2^{n}} \leq 1/2$ and has edge-boundary satisfying
\[|\partial A| \leq L |A| \log_{2}(2^{n}/|A|),\]
then there exists a subcube \(C \subset \{0,1\}^n\) with codimension at most \(L' \log_{2}(1/p)\), such that
\[\frac{|A \cap C|}{|C|} \geq (1+\delta)p.\]
\end{conjecture}

{\em Acknowledgements:} The author would like to thank Alex Samorodnitsky for much helpful advice, and also Ehud Friedgut, Imre Leader, and Benny Sudakov for valuable discussions.

\end{document}